\documentclass[reqno]{amsart} 
\usepackage[T1]{fontenc}
\usepackage[english]{babel}
\usepackage{amssymb,bbm,enumerate}
\usepackage{color}
\usepackage[linktocpage=true,colorlinks=true, linkcolor=blue, citecolor=red, urlcolor=green]{hyperref}

\usepackage{xypic}

\usepackage{tikz}

\renewcommand{\phi}{\varphi}

\def\ch{\mathrm{ch}}
\def\cl{\mathrm{cl}}

\def\dd{\partial}

\def\Ga{\Gamma}
\def\la{\lambda}

\newcommand{\mc}[1]{\mathcal{#1}}

\newcommand{\mb}[1]{\mathbb{#1}}

\DeclareMathOperator{\Hom}{Hom}
\DeclareMathOperator{\End}{End}

\DeclareMathOperator{\ad}{ad}

\DeclareMathOperator{\Ker}{Ker}

\DeclareMathOperator{\Res}{Res}

\DeclareMathOperator{\fil}{F}
\DeclareMathOperator{\gr}{gr}

\makeatletter
\def\smallunderbrace#1{\mathop{\vtop{\m@th\ialign{##\crcr
   $\hfil\displaystyle{#1}\hfil$\crcr
   \noalign{\kern3\p@\nointerlineskip}%
   \tiny\upbracefill\crcr\noalign{\kern3\p@}}}}\limits}

\theoremstyle{plain}
\newtheorem{theorem}{Theorem}[section]
\newtheorem*{theorem*}{Theorem}
\newtheorem{lemma}[theorem]{Lemma}
\newtheorem{proposition}[theorem]{Proposition}

\theoremstyle{definition}
\newtheorem{definition}[theorem]{Definition}
\newtheorem{example}[theorem]{Example}

\theoremstyle{remark}
\newtheorem{remark}[theorem]{Remark}

\numberwithin{equation}{section}

\definecolor{light}{gray}{.9}

\begin{document}

\title{Chiral vs classical operad}

\author{Bojko Bakalov}
\address{Department of Mathematics, North Carolina State University,
Raleigh, NC 27695, USA}
\email{bojko\_bakalov@ncsu.edu}
%
\author{Alberto De Sole}
\address{Dipartimento di Matematica, Sapienza Universit\`a di Roma,
P.le Aldo Moro 2, 00185 Rome, Italy}
\email{desole@mat.uniroma1.it}
\urladdr{www1.mat.uniroma1.it/$\sim$desole}
\author{Reimundo Heluani}
\address{IMPA, Rio de Janeiro, Brasil}
\email{heluani@impa.br}
\author{Victor G. Kac}
\address{Department of Mathematics, MIT,
77 Massachusetts Ave., Cambridge, MA 02139, USA}
\email{kac@math.mit.edu}

\subjclass{
Primary 18D50; Secondary 17B63, 17B69, 05C25
}


\begin{abstract}
We establish an explicit isomorphism between the associated graded of the filtered
chiral operad and the classical operad,
which is important for computing the cohomology of vertex algebras.
\end{abstract}
\keywords{
Chiral and classical operads,
$\Gamma$-residue and $\Gamma$-Fourier transform.
}

\maketitle


\pagestyle{plain}

\section{Introduction}\label{sec:1}

This is the second in a series of papers aimed at computing the cohomology of vertex algebras.
In our first paper \cite{BDSHK}, for a vector superspace $V$ with an even endomorphism $\partial$, we introduced the chiral operad $P^{\ch}(V)$. 
This is an explicit description in local coordinates of the chiral operad of Beilinson and Drinfeld \cite{BD04},
associated to a $\mc D$-module on a smooth algebraic curve $X$,
where the geometric language of $\mc D$-modules
is replaced by the linear algebraic language of (integrals of) lambda-brackets.
We are taking $X=\mb A^1$ and the $\mc D$-module translation equivariant.

The operad $P^{\ch}(V)$ ``governs'' vertex algebra cohomology in the following sense. 
To each vector superspace $V$ 
over a field $\mb F$ of characteristic zero,
with an even endomorphism $\partial$,
it canonically associates a $\mb Z$-graded Lie superalgebra
\begin{equation}\label{intro:1.1}
W^{\ch}(V)
=
\bigoplus_{k=-1}^\infty W_k^{\ch}(V)
\,, \qquad\text{ where}\quad
W^{\ch}_k(V)=P^{\ch}(V)(k+1)^{S_{k+1}} \,,
\end{equation}
such that
\begin{equation}\label{intro:1.2}
W^{\ch}_{-1}(V)=V/\partial V
\,,\qquad
W^{\ch}_0(V)=\End_{\mb F[\partial]} V
\,.
\end{equation}
This Lie superalgebra is an explicit description in local coordinates of the Lie superalgebra constructed by Tamarkin in \cite{Tam02} in his study of deformations of chiral algebras, in the particular case of translation equivariant chiral algebras on $\mathbb{A}^1$. 

The space $W_k^{\ch}(V)$ consists of all elements from $P^{\ch}(V)(k+1)$ that are invariant under the action of the symmetric group $S_{k+1}$, and
the Lie bracket on $W^{\ch}(V)$ is defined via the $\circ_i$-products of the operad $P^{\ch}(V)$.
For the construction of the $\mb Z$-graded Lie superalgebra associated
to an arbitrary linear operad, see \cite{Tam02,LV12,BDSHK}.

An odd element $X\in W_1^{\ch}(\Pi V)$ satisfying $[X,X]=0$,
where $\Pi V$ stands for $V$ with the reversed parity,
defines on $V$ the structure a non-unital vertex algebra.
Consequently, $(W^{\ch}(\Pi V),\ad X)$ is a differential graded Lie superalgebra
whose cohomology is the cohomology of the vertex algebra $V$ defined by $X$,
with coefficients in the adjoint module.
Alternatively, a non-unital vertex algebra structure on $V$ is equivalent to a morphism of operads $\mathcal{L}ie \rightarrow P^{\ch}(V)$ \cite[Sec.\ 3.3.3]{BD04}. 
The image of the binary operation $[\,,] \in \mathcal{L}ie(2)$ is given by the (parity shifted) operation $X$. 

Let us recall the definition of the operad $P^{\ch}(V)$ associated to a vector superspace $V$
with an even endomorphism $\partial$.
For a non-negative integer $n$, let 
$$
\mc O^{\star T}_n=\mb F[(z_i-z_j)^{\pm1}\,|\,1\leq i<j\leq n]
\,,
$$
the algebra of Laurent polynomials in $z_i-z_j$.
Denote by $\partial_i$ the endomorphism of $V^{\otimes n}$
acting as $\partial$ on the $i$-th factor.
Introduce the superspace
$$
V_n=V[\lambda_1,\dots,\lambda_n]/\langle\partial+\lambda_1+\dots+\lambda_n\rangle
\,,
$$
where all variables $\lambda_i$ have even parity
and $\langle\Phi\rangle$ stands for the image of the endomorphism $\Phi$.
The superspace of $n$-ary chiral operations $P^{\ch}(n):=P^{\ch}(V)(n)$
is defined as the set of all linear maps \cite[Eq.\ (6.11)]{BDSHK}
\begin{equation} \label{eq:chiral-1}
X\colon V^{\otimes n}\otimes\mc O^{\star T}_n \to V_n
\,,\qquad
v\otimes f \mapsto X_{\lambda_1,\dots,\lambda_n}(v\otimes f)
\,,
\end{equation} 
satisfying the following two sesquilinearity axioms ($i,j=1,\dots,n$):
\begin{equation} \label{eq:chiral-2}
\begin{aligned}
X_{\lambda_1,\dots,\lambda_n} ( v \otimes \partial_{z_i} f )
&= X_{\lambda_1,\dots,\lambda_n} ( (\partial_i+\la_i) v \otimes f ) \,, \\
X_{\lambda_1,\dots,\lambda_n} (v \otimes (z_i-z_j)f)
&= (\partial_{\lambda_j}-\partial_{\lambda_i}) X_{\lambda_1,\dots,\lambda_n} (v \otimes f) \,.
\end{aligned}
\end{equation} 
Note that $P^{\ch}(n)=W^{\ch}_{n-1}(V)$ is given by \eqref{intro:1.2} for $n=0,1$.
In \cite{BDSHK}, we also introduced the action of $S_n$ on $P^{\ch}(n)$
and the $\circ_i$-products to make $P^{\ch}(V)$ into an operad.

Now suppose that $V$ is equipped with an increasing filtration of $\mb F[\partial]$-submodules
\begin{equation}\label{intro:1.4}
\fil^{-1}V=\{0\}
\,\subset\,
\fil^0V
\,\subset\,
\fil^1V
\,\subset\,
\fil^2V
\,\subset\,
\cdots
\,\subset\,
V
\,.
\end{equation}
Taking the increasing filtration of $\mc O^{\star T}_n$ by the number of divisors,
the filtration \eqref{intro:1.4}
induces an increasing filtration on $V^{\otimes n}\otimes\mc O^{\star T}_n$.
The latter induces a decreasing filtration on $P^{\ch}(n)$.
The associated graded pieces $\gr^r P^{\ch}(n)$, $r\geq0$,
form a graded operad denoted by $\gr P^{\ch}$.

On the other hand, in \cite{BDSHK} we introduced the operad $P^\cl(V)$, which ``governs'' 
the Poisson vertex algebra cohomology in a similar way. 
Moreover, assuming that $V$ is $\mb Z$-graded by $\mb F[\partial]$-submodules, we have 
the associated $\mb Z$-grading on the space of $n$-ary operations $P^{\cl}(n) := P^{\cl}(V)(n)$ 
$$
P^\cl(n) = \bigoplus_{r\in\mb Z} \gr^r P^\cl(n) \,.
$$
Next, assuming that $V$ is endowed with the filtration \eqref{intro:1.4}, we have the linear map
\begin{equation}\label{intro:1.5}
\gr^r P^\ch(V)(n) \to \gr^r P^\cl(\gr V)(n) \,, \qquad r\ge0\,.
\end{equation}
These constructions are recalled in Section \ref{sec:3}.
We proved in \cite{BDSHK} that the map \eqref{intro:1.5} is an injective morphism of operads.
The surjectivity of this map was proposed as a conjecture.

The main result of the present paper is that the map \eqref{intro:1.5} is an isomorphism, provided that 
the filtration \eqref{intro:1.4} is induced by a grading by $\mb F[\partial]$-modules (Theorem \ref{thm:main}). 
In fact, we construct explicitly a map, inverse to \eqref{intro:1.5}, using the notions of $\Gamma$-residue 
and $\Gamma$-Fourier transform introduced in Section \ref{sec:4}. 

Theorem \ref{thm:main} is important since it allows to compare the vertex algebra and Poisson 
vertex algebra cohomology. For example, using the obvious fact that this theorem holds 
(without any assumptions) for $n=0,1$ and the results of \cite{DSK12,DSK13} 
on variational Poisson cohomology, we calculated in \cite{BDSHK} the $0$-th and $1$-st cohomology 
of the vertex algebra of free bosons, computing thereby its Casimirs and derivations.
The connection between the classical and variational Poisson cohomology is discussed 
in the forthcoming paper \cite{BDSHKV19}.

Our operad $P^\cl(V)$ was shown to be related to Beilinson and Drinfeld's operad 
of classical operations in \cite[Appendix]{BDSHK}. The isomorphism of Theorem \ref{thm:main} 
is stated in \cite[Sec.\ 3.2.5]{BD04} for the trivial filtration of $V$ in the geometric context under the assumption 
that the corresponding $\mathcal{D}$-module is projective. 

As pointed out by a referee, a more appropriate terminology would have been
the operads of chiral and of classical operations for $P^{\ch}(V)$ and $P^{\cl}(V)$, respectively.
We opted for the shorter names chiral and classical operads.

Throughout the paper the base field $\mb F$ has characteristic $0$.

\subsubsection*{Acknowledgments} 

This research was partially conducted during the authors' visits 
to RIMS in Kyoto and to the University of Rome La Sapienza.
We are grateful to these institutions for their kind hospitality.
We thank the referees for their valuable comments.
The first author is supported in part by a Simons Foundation grant 584741.
The second author was partially supported by the national PRIN fund n.\ 2015ZWST2C$\_$001
and the University funds n. RM116154CB35DFD3 and RM11715C7FB74D63.
The third author is partially supported by the Bert and Ann Kostant fund.

\section{The chiral operad}\label{sec:2}

In this section, we recall the definition of the chiral operad $P^\ch(V)$ from \cite[Sec.\ 6]{BDSHK}.

\subsection{The spaces $\mc O_{n}^{\star T}$}\label{sec:ostart}
Here and further, we will consider rational functions in the variables $z_1,z_2,\dots$ and use the shorthand notation $z_{ij}=z_i-z_j$.
For a fixed positive integer $n$, we denote by $\mc O_{n} = \mb F[z_1,\dots,z_n]$ the algebra of polynomials, and by
$$
\mc O_{n}^T
=
\mb F[z_{ij}]_{1\leq i<j\leq n}
= \Ker \sum_{i=1}^n \partial_{z_i}
$$
the subalgebra of translation invariant polynomials.
Let $\mc O_{n}^{\star}$ be the localization of $\mc O_{n}$ with respect to the diagonals $z_i=z_j$ for $i\neq j$, i.e.,
$$
\mc O_{n}^{\star}
=
\mb F[z_1,\dots,z_n][z_{ij}^{-1}]_{1\leq i<j\leq n},
$$
and let 
$$
\mc O_{n}^{\star T}
=\mb F[z_{ij}^{\pm 1}]_{1\leq i<j\leq n}.
$$
We also set $\mc O_{0} = \mc O_{0}^T = \mc O_{0}^{\star}=\mc O_{0}^{\star T}=\mb F$.
Note that $\mc O_{1} = \mc O_{1}^{\star} = \mb F[z_1]$ 
and $\mc O_{1}^T = \mc O_{1}^{\star T}=\mb F$.
At times we will denote $\mc O^{\star}_n=\mc O^{\star}_n(z_1,\dots,z_n)$,
if we want to specify the variables $z_1,\dots,z_n$.

We introduce an increasing filtration of $\mc O_{n}^{\star}$ given by the number of divisors:
\begin{equation}\label{fil1}
\begin{split}
\fil^{-1} \mc O_{n}^{\star} &= \{0\} \subset \fil^0 \mc O_{n}^{\star} = \mc O_{n} \subset
\fil^1 \mc O_{n}^{\star} = \sum_{i<j} \mc O_{n} [z_{ij}^{-1}] \subset \\
\cdots&\subset \fil^r \mc O_{n}^{\star} = \sum \mc O_{n} [z_{i_1,j_1}^{-1},\dots, z_{i_r,j_r}^{-1}] \subset\cdots
\subset  \fil^{n-1} \mc O_{n}^{\star} = \mc O_{n}^{\star}.
\end{split}
\end{equation}
In other words, the elements of $\fil^r \mc O_{n}^{\star}$ are sums of rational functions with
at most $r$ poles each, not counting multiplicities.
The fact that $\fil^{n-1} \mc O_{n}^{\star} = \mc O_{n}^{\star}$ was proved in \cite{BDSHK}
(it is a consequence of the proof of Lemma 8.4 there). 
By restriction, we have the induced increasing filtration 
$$
\fil^r\mc O^{\star T}_n=\fil^r\mc O^{\star}_n\cap\mc O^{\star T}_n
\,.
$$

\subsection{The operad $P^\ch (V)$}\label{sec:wch}

Let $V=V_{\bar 0}\oplus V_{\bar 1}$ be a vector superspace endowed
with an even endomorphism $\partial$. For every $i=1,\dots,n$, we will denote by $\partial_i$ the action of $\partial$ on the $i$-th factor of the tensor power $V^{\otimes n}$:
\begin{equation}\label{ddi}
\partial_i v = v_1 \otimes\cdots\otimes \partial v_i  \otimes\cdots\otimes v_n \quad\text{for}\quad
v = v_1 \otimes\cdots\otimes v_n \in V^{\otimes n}.
\end{equation}
Consider the space 
\begin{equation}\label{20160722:eq1}
V[\lambda_1,\dots,\lambda_n]\big/\big\langle\partial+\lambda_1+\dots+\lambda_n\big\rangle
\,,
\end{equation}
where here and further, $\langle\Phi\rangle$ denotes the image of an endomorphim $\Phi$.
 
The space of $n$-ary \emph{chiral operations} $P^\ch(n):=P^{\ch}(V)(n)$ 
is defined as the set of all linear maps \cite[Eq.\ (6.11)]{BDSHK}
\begin{equation}\label{20160629:eq2-c}
\begin{split}
X\colon
V^{\otimes n}\otimes\mc O_{n}^{\star T}
&\to
V[\lambda_1,\dots,\lambda_n]\big/\big\langle\partial+\lambda_1+\dots+\lambda_n\big\rangle
\,,\\
\vphantom{\Big(}
v_1 \otimes\dots\otimes v_n\otimes &f(z_1,\dots,z_n)
\mapsto
X_{\lambda_1,\dots,\lambda_n} (v_1\otimes\dots\otimes v_n \otimes f)
\\
&=
X_{\lambda_1,\dots,\lambda_n}^{z_1,\dots,z_n} (v_1\otimes\dots\otimes v_n \otimes f(z_1,\dots,z_n))
\,,
\end{split}
\end{equation}
satisfying the following two \emph{sesquilinearity} conditions:
\begin{align}\label{20160629:eq4a}
X_{\lambda_1,\dots,\lambda_n} ( v \otimes \partial_{z_i} f )
&= X_{\lambda_1,\dots,\lambda_n} ( (\partial_i+\la_i) v \otimes f ) \,, \\
\label{20160629:eq4b}
X_{\lambda_1,\dots,\lambda_n} (v \otimes z_{ij}f)
&= (\partial_{\lambda_j}-\partial_{\lambda_i}) X_{\lambda_1,\dots,\lambda_n} (v \otimes f) \,.
\end{align}
For example, we have:
\begin{align}\label{pch0}
P^\ch(0) &= \Hom_{\mb F} (\mb F, V/\langle\dd\rangle) \cong V/\dd V, \\
\label{pch1}
P^\ch(1) &= \Hom_{\mb F[\dd]} (V, V[\la_0]/\langle\dd+\la_0\rangle) \cong \End_{\mb F[\dd]} (V).
\end{align}
The $\mb Z/2\mb Z$-grading of the superspace $P^\ch(n)$ is induced 
by that of the vector superspace $V$, where $\mc O_{n}^{\star T}$ and all
variables $\lambda_i$ are considered even.
%

%
\begin{remark} The spaces $V^{\otimes n} \otimes \mc O_n^{\star T}$ and \eqref{20160722:eq1} 
are canonically modules over the algebra $\mathcal{D}$ of translation invariant differential operators 
in $n$-variables. 
Taking the quotient in \eqref{20160722:eq1} means that the sum of coordinate vector fields
is the diagonal vector field.
Equations \eqref{20160629:eq4a}--\eqref{20160629:eq4b} are equivalent to stating 
that $X$ is a morphism of $\mathcal{D}$-modules. 
\end{remark}

One can also define an action of the symmetric group and compositions of chiral operations, 
turning $P^\ch(V)$ into an operad (see \cite[Eq.\ (6.25)]{BDSHK}). 
However, these structures will not be needed in the present paper,
hence we do not recall their definition.

\subsection{Filtration of $P^\ch(V)$}\label{sec:pchfil}
Now suppose that $V$ is equipped with an increasing filtration
of $\mb F[\partial]$-submodules
\begin{equation}\label{eq:last3}
\fil^{-1}V=\{0\}
\,\subset\,
\fil^0V
\,\subset\,
\fil^1V
\,\subset\,
\fil^2V
\,\subset\,
\cdots
\,\subset\,
V
\,.
\end{equation}
Since $\mc O_{n}^{\star T}$ is also filtered by \eqref{fil1}, we obtain
an increasing filtration on the tensor products
$$
\fil^s\big(V^{\otimes n}\otimes \mc O_{n}^{\star T}\big)
=
\sum_{s_1+\dots+s_{n}+p\leq s}
\fil^{s_1}V
\otimes
\cdots
\otimes
\fil^{s_n}V
\otimes
\fil^{p}\mc O_{n}^{\star T}
\,,
$$
if $s\geq0$, and $\fil^s(V^{\otimes n}\otimes \mc O_{n}^{\star T})=\{0\}$ if $s<0$.
%
%
%
This induces a decreasing filtration of $P^\ch(n)$, where $\fil^r P^\ch(n)$ for $r\in\mb Z$ is defined 
as the set of all elements $X$ such that
\begin{equation}\label{fil4-ref}
X\big( 
\fil^s(V^{\otimes n} \otimes \mc O_{n}^{\star T})
\big)
\subset
(\fil^{s-r}V)[\lambda_1,\dots,\lambda_n]/\langle\partial+\lambda_1+\dots+\lambda_n\rangle
\,,
\end{equation}
for every $s$.
Then, as usual, the associated graded spaces are defined by
\begin{equation}\label{grr}
\gr^r P^\ch(n)
=
\fil^r P^\ch(n) / \fil^{r+1} P^\ch(n).
\end{equation}
In fact, the composition maps are compatible with the filtration \eqref{fil4-ref}
and, therefore, the associated graded \eqref{grr} is a graded operad (see \cite[Prop.\ 8.1]{BDSHK}).

\section{The classical operad}\label{sec:3}

Here we recall the definition of the classical operad $P^\cl(V)$ from \cite[Sec.\ 10]{BDSHK}.

\subsection{$n$-graphs}\label{sec:6a.1}

For a positive integer $n$, we define an $n$-\emph{graph}
as a graph $\Gamma$ with $n$ vertices labeled by $1,\dots,n$
and an arbitrary collection $E(\Gamma)$ of oriented edges.
We denote 
by $\mc G(n)$ the collection of all $n$-graphs
without tadpoles,
and by $\mc G_0(n)$ the collection of all \emph{acyclic} $n$-graphs,
i.e., $n$-graphs that have 
no cycles (including tadpoles and multiple edges).
For example, $\mc G_0(1)$ consists of the graph with a single vertex labeled $1$ and no edges,
and $\mc G_0(2)$ consists of three graphs:
\begin{equation}\label{eq:2-graphs}
\begin{array}{l}
\begin{tikzpicture}
\draw (0.5,1) circle [radius=0.07];
\node at (0.5,0.7) {1};
\draw (1.5,1) circle [radius=0.07];
\node at (1.5,0.7) {2};
\node at (2,0.85) {,};
\draw (3.7,1) circle [radius=0.07];
\node at (3.7,0.7) {1};
\draw (4.7,1) circle [radius=0.07];
\node at (4.7,0.7) {2};
\draw[->] (3.8,1) -- (4.6,1);
\node at (5,0.85) {,};
\draw (7,1) circle [radius=0.07];
\node at (7,0.7) {1};
\draw (8,1) circle [radius=0.07];
\node at (8,0.7) {2};
\draw[<-] (7.1,1) --(7.9,1);
\node at (8.5,0.85) {.};
\end{tikzpicture}
\\
E(\Gamma)=\emptyset
\,\,\,\,\,,\,\,\,\,\qquad
E(\Gamma)\!=\!\{1\!\to\!2\}
\,\,,\,\,\qquad
E(\Gamma)\!=\!\{2\!\to\!1\}
\end{array}
\end{equation}
By convention, we also let $\mc G_0(0)=\mc G(0)=\{\emptyset\}$ be the set consisting of a single element (the empty graph with $0$ vertices).

A graph $L$ will be called a \emph{line} if its set of edges is of the form $\{i_1\to i_2,\,i_2\to i_3,\dots,\,i_{n-1}\to i_n\}$ where $\{i_1,\dots,i_n\}$ is a permutation of $\{1,\dots,n\}$:
\begin{equation}\label{eq:line}
\begin{tikzpicture}
\node at (-0.2,1) {$L=$};
\draw (0.5,1) circle [radius=0.07];
\node at (0.5,0.6) {$i_1$};
\draw[->] (0.6,1) -- (0.9,1);
\draw (1,1) circle [radius=0.07];
\node at (1,0.6) {$i_2$};
\draw[->] (1.1,1) -- (1.4,1);
\node at (1.7,1) {$\cdots$};
\draw[->] (1.9,1) -- (2.2,1);
\draw (2.3,1) circle [radius=0.07];
\node at (2.3,0.6) {$i_n$};
\node at  (2.9,0.8) {.};
\end{tikzpicture}
\end{equation}
An \emph{oriented cycle} $C$ in a graph $\Gamma$ is, by definition, 
a collection of edges of $\Gamma$ forming a closed sequence (possibly with self intersections):
\begin{equation}\label{20170823:eq4a}
C=\{i_1\to i_2,\,i_2\to i_3,\dots,\,i_{s-1}\to i_s,\,i_s\to i_1\}\subset E(\Gamma)
\,.
\end{equation}

There is a natural (left) action of the symmetric group $S_n$
on the set $\mc G(n)$ of $n$-graphs, which preserves the subset $\mc G_0(n)$ of acyclic graphs.
Given $\Gamma\in\mc G(n)$ and $\sigma\in S_n$,
we define $\sigma(\Gamma)$ to be the same graph as $\Gamma$,
but with the vertex that was labeled $1$ relabeled as $\sigma(1)$,
the vertex $2$ relabeled as $\sigma(2)$,
and so on up to the vertex $n$ now relabeled as $\sigma(n)$.
For example, if $L_0$ is the line with edges $\{1\to2,2\to3,\dots,n-1\to n\}$ and $\sigma\in S_n$,
then $\sigma(L_0)=L$ is the line \eqref{eq:line} where $i_k=\sigma(k)$.

\subsection{The operad $P^\cl(V)$}\label{sec:6.2}

As before, let $V=V_{\bar 0}\oplus V_{\bar 1}$ be a vector superspace endowed
with an even endomorphism $\partial$.
As a vector superspace, $P^\cl(n):=P^\cl(V)(n)$ 
is defined as the vector superspace (with the pointwise addition and scalar multiplication)
of all maps 
\begin{align}\label{20170614:eq4}
Y\colon
\mc G(n)\times V^{\otimes n}
&\to
V[\lambda_1,\dots,\lambda_n] / \langle \dd+\lambda_1+\dots+\lambda_n \rangle
\,, \\ \label{20170614:eq5}
\Gamma \times v &\mapsto
Y^{\Gamma}_{\lambda_1,\dots,\lambda_n}(v) \,,
\end{align}
which depend linearly on 
$v 
\in V^{\otimes n}$,
and satisfy
the {cycle relations}
and {sesquilinearity conditions} described below.
The $\mb Z/2\mb Z$-grading of the superspace $P^\cl(n)$ is induced 
by that of the vector superspace $V$,
by letting $\Gamma$ and the variables $\lambda_i$ be even.

The \emph{cycle relations} state that if an $n$-graph $\Gamma\in\mc G(n)$ contains an oriented cycle
$C\subset E(\Gamma)$, then:
\begin{equation}\label{eq:cycle}
Y^{\Gamma} =0 
\,,\qquad
\sum_{e\in C} Y^{\Gamma\backslash e} =0 
\,,
\end{equation}
where $\Gamma\backslash e\in\mc G(n)$ is the graph obtained from $\Gamma$ 
by removing the edge $e$ and keeping the same set of vertices.
In particular, applying the second cycle relation \eqref{eq:cycle} for an oriented cycle of length $2$,
we see that changing the orientation of a single edge of $\Gamma\in\mc G(n)$
amounts to a change of sign of $Y^{\Gamma}$.

To write the {sesquilinearity conditions}, let us first introduce some notation.
For a graph $G$ with a set of vertices labeled by a subset $I\subset\{1,\dots,n\}$,
we let
\begin{equation}\label{com3}
\la_G = \sum_{i\in I} \la_i \,, 
\qquad \dd_G = \sum_{i\in I} \dd_i \,, 
\end{equation}
where as before $\dd_i$ denotes the action of $\dd$ on the $i$-th factor in $V^{\otimes n}$
(see \eqref{ddi}).
Then for every connected component $G$ of $\Gamma\in\mc G(n)$ with a set of vertices $I$,
we have two \emph{sesquilinearity conditions}:
\begin{align}\label{eq:sesq1}
(\partial_{\lambda_j} - \partial_{\lambda_i})
Y^{\Gamma}_{\lambda_1,\dots,\lambda_n}(v) &= 0
\quad\text{for all}\quad
i,j\in I \,,
\\ \label{eq:sesq2}
Y^{\Gamma}_{\lambda_1,\dots,\lambda_n}
\bigl( (\partial_G+\la_G)v \bigr)
&=0 \,, \qquad v\in V^{\otimes n} \,.
\end{align}
The first condition \eqref{eq:sesq1} means that the polynomial 
$Y^{\Gamma}_{\lambda_1,\dots,\lambda_n}(v)$
is a function of the variables $\lambda_{\Gamma_\alpha}$, where 
the $\Gamma_\alpha$'s are the connected components of $\Gamma$, 
and not of the variables $\lambda_1,\dots,\lambda_n$ separately.

%
%
In \cite[Eq.\ (10.11)]{BDSHK}, we also defined the action of the symmetric group
and compositions of maps in $P^\cl(V)$, turning it into an operad. 
However, these structures will not be needed in the present paper.

\subsection{Grading of $P^\cl(V)$}\label{sec:grpcl}

Suppose now that $V=\bigoplus_{t\in\mb Z} \gr^t V$ is graded by $\mb F[\partial]$-sub\-modules,
and consider the induced grading of the tensor powers $V^{\otimes n}$:
$$
\gr^t V^{\otimes n}
=
\sum_{t_1+\dots+t_{n}=t}
\gr^{t_1}V
\otimes
\cdots
\otimes
\gr^{t_n}V
\,.
$$
Then $P^\cl(V)$
has a grading defined as follows:
$Y\in\gr^r P^\cl(n)$ if
\begin{equation}\label{pclgrading}
Y^\Gamma_{\lambda_1,\dots,\lambda_n}(\gr^t V^{\otimes n})
\,\subset\,
(\gr^{s+t-r}V)[\lambda_1,\dots,\lambda_n]/
\langle\partial+\lambda_1+\dots+\lambda_n\rangle
\end{equation}
for every graph $\Gamma\in\mc G(n)$ with $s$ edges
(see \cite[Rem.\ 10.2]{BDSHK}).
%

\subsection{The map from $\gr P^\ch(V)$ to $P^\cl(V)$}\label{sec:grpchpcl}

For a graph $\Gamma\in\mc G(n)$ with a set of edges $E(\Ga)$, we introduce the function
\begin{equation}\label{gr1}
p_\Ga = p_\Ga(z_1,\dots,z_n) = \prod_{(i\to j) \in E(\Ga)} z_{ij}^{-1} \,, 
\qquad z_{ij}=z_i-z_j \,.
\end{equation}
Note that $p_\Ga \in \fil^s \mc O_{n}^{\star T}$ if $\Ga$ has $s$ edges. 
\begin{lemma}\label{lem:pgamma}
Let\/ $\Ga\in\mc G(n)$ be a graph with $s$ edges, containing a cycle $C\subset E(\Gamma)$.
Then:
\begin{enumerate}[(a)]
\item
$p_\Ga \in \fil^{s-1} \mc O_{n}^{\star T}$;
\item
$\sum_{e\in C}p_{\Gamma\backslash e}=0$.
\end{enumerate}
\end{lemma}
\begin{proof}
The proofs of both statements are contained 
in the proof of \cite[Lem.\ 8.4]{BDSHK}.
\end{proof}
Let $V$ be filtered by $\mb F[\partial]$-submodules as in \eqref{eq:last3}. 
Then we have the filtered operad $P^\ch(V)$ associated to $V$ 
and the graded operad $P^\cl(\gr V)$ associated to the graded superspace $\gr V$. 
These two operads are related as follows \cite[Sec.\ 8]{BDSHK}.

Let $X\in\fil^r P^\ch(V)(n)$ and $\Gamma\in\mc G(n)$ be a graph with $s$ edges. Then for every $v\in\fil^t V^{\otimes n}$, we have
$v\otimes p_\Ga \in \fil^{s+t} ( V^{\otimes n} \otimes \mc O_{n}^{\star T} )$ and, by \eqref{fil4-ref},
\begin{equation}\label{gr2}
X_{\la_1,\dots,\la_n} (v\otimes p_\Ga) 
\in (\fil^{s+t-r} V)[\lambda_1,\dots,\lambda_n]/\langle\partial+\lambda_1+\dots+\lambda_n\rangle \,.
\end{equation}
We define $Y\in\gr^r P^\cl(\gr V)(n)$ by:
\begin{equation}\label{gr3}
\begin{split}
Y^\Ga_{\la_1,\dots,\la_n} \bigl( v&+\fil^{t-1} V^{\otimes n} \bigr)
= X_{\la_1,\dots,\la_n} (v\otimes p_\Ga) \\
&+ (\fil^{s+t-r-1} V)[\lambda_1,\dots,\lambda_n]/\langle\partial+\lambda_1+\dots+\lambda_n\rangle \\
&\in (\gr^{s+t-r} V)[\lambda_1,\dots,\lambda_n]/\langle\partial+\lambda_1+\dots+\lambda_n\rangle \,.
\end{split}
\end{equation}
Clearly, the right-hand side depends only 
on the image $\bar v = v+\fil^{t-1} V^{\otimes n} \in \gr^t V^{\otimes n}$ 
and not on the choice of representative $v\in\fil^t V^{\otimes n}$. 
We write \eqref{gr3} simply as
\begin{equation}\label{gr4}
Y^\Ga_{\la_1,\dots,\la_n} (\bar v)
= \overline{ X_{\la_1,\dots,\la_n} (v\otimes p_\Ga) } \,.
\end{equation}
The fact that $Y\in\gr^r P^\cl(\gr V)(n)$ was proved in \cite[Cor.\ 8.8]{BDSHK}. 

If $X\in\fil^{r+1} P^\ch(V)(n)$, then the right-hand side of \eqref{gr3} (or \eqref{gr4}) vanishes.
Thus, \eqref{gr3} defines a map
\begin{equation}\label{eq:map}
\gr^r P^\ch(V)(n)\,\to\,\gr^r P^\cl(\gr V)(n)
\,,\qquad 
\bar X=X+\fil^{r+1} \mapsto Y
\,.
\end{equation}
\begin{theorem}[{\cite{BDSHK}}] 
The map \eqref{eq:map} is an injective homomorphism of graded operads. 
\end{theorem}
We will not need the full statement here (see \cite[Thm.\ 10.12]{BDSHK}), but let us observe that \eqref{gr3} is compatible with the actions of the symmetric group $S_n$. 
In \cite[Rem.\ 10.15]{BDSHK}, we also posed the question whether the map \eqref{eq:map}
is an isomorphism.
The main result of the present paper is Theorem \ref{thm:main} below,
which says, in particular, that this is indeed the case under the assumption that $V$
is graded as an $\mb F[\partial]$-module.

\section{$\Gamma$-residues and $\Gamma$-Fourier transform}\label{sec:4}

\subsection{Lines}\label{sec:4.1}

Given a positive integer $n$, let $\mc L(n)\subset\mc G(n)$
be the set of graphs that are disjoint unions of lines,
i.e., graphs of the following form:
\begin{equation}\label{eq:Gamma}
\begin{tikzpicture}
\node at (-0.2,1) {$\Gamma=$};
\draw (0.5,1) circle [radius=0.07];
\node at (0.5,0.6) {$i^1_1$};
\draw[->] (0.6,1) -- (0.9,1);
\draw (1,1) circle [radius=0.07];
\node at (1,0.6) {$i^1_2$};
\draw[->] (1.1,1) -- (1.4,1);
\node at (1.7,1) {$\cdots$};
\draw[->] (1.9,1) -- (2.2,1);
\draw (2.3,1) circle [radius=0.07];
\node at (2.3,0.6) {$i^1_{k_1}$};
\draw (3,1) circle [radius=0.07];
\node at (3,0.6) {$i^2_1$};
\draw[->] (3.1,1) -- (3.4,1);
\draw (3.5,1) circle [radius=0.07];
\node at (3.5,0.6) {$i^2_2$};
\draw[->] (3.6,1) -- (3.9,1);
\node at (4.2,1) {$\cdots$};
\draw[->] (4.4,1) -- (4.7,1);
\draw (4.8,1) circle [radius=0.07];
\node at (4.8,0.6) {$i^2_{k_2}$};
\node at (5.45,1) {$\cdots$};
\draw (6,1) circle [radius=0.07];
\node at (6,0.6) {$i^p_1$};
\draw[->] (6.1,1) -- (6.4,1);
\draw (6.5,1) circle [radius=0.07];
\node at (6.5,0.6) {$i^p_2$};
\draw[->] (6.6,1) -- (6.9,1);
\node at (7.2,1) {$\cdots$};
\draw[->] (7.4,1) -- (7.7,1);
\draw (7.8,1) circle [radius=0.07];
\node at (7.8,0.6) {$i^p_{k_p}$};
\node at (9.7,1) {$=L_1\sqcup L_2\sqcup\dots\sqcup L_p\,,$};
\end{tikzpicture}
\end{equation}
where $k_1,\dots,k_p\geq1$ are such that $k_1+\dots+k_p=n$,
and the set of indices $\{i^a_b\}$ is a permutation of $\{1,\dots,n\}$ such that
\begin{equation}\label{eq:Gcond}
i^1_1=1\,<\,i^2_1\,<\,\cdots\,<\,i^p_1
\,,\qquad
i^\ell_1=\min\{i^\ell_1,\dots,i^\ell_{k_\ell}\}
\,,\,\,\ell=1,\dots,p
\,.
\end{equation}
In \eqref{eq:Gamma}, $L_r$ denotes the $r$-th connected component of $\Gamma$
(which is a connected oriented line of length $k_r$).
For example, 
when $k_r=1$ the line $L_r$ consists of the single vertex indexed $i^r_1$.
We also denote by $\mc L(n,p)\subset\mc L(n)$
the subset of graphs $\Gamma$ as in \eqref{eq:Gamma}
with the fixed number $p$ of connected components.

Consider the vector space $\mb F\mc G(n)$ linearly spanned by the set $\mc G(n)$.
The \emph{cycle relations} in $\mb F\mc G(n)$ are the following elements:
\begin{enumerate}[(i)]
\item
all graphs $\Gamma\in\mc G(n)$ containing a cycle;
\item
all linear combinations of the form
$\sum_{e\in C}\Gamma\backslash e$,
for $\Gamma\in\mc G(n)$
and all oriented cycles $C\subset E(\Gamma)$.
\end{enumerate}
Note that if we reverse an arrow in a graph $\Gamma\in\mc G(n)$,
we obtain, modulo cycle relations, the element $-\Gamma\in\mb F\mc G(n)$.
\begin{lemma}\label{prop:basis}
The set $\mc L(n)$ spans the space $\mb F\mc G(n)$ modulo
the cycle relations.
\end{lemma}
\begin{proof}
Let $\Gamma\in\mc G(n)$.
First, we claim that, modulo cycle relations,
we can assume that the vertex $1$ is a leaf,
i.e., there is no more than one edge in or out of it.
Indeed, if there are $\ell\geq2$ edges in or out of $1$, then
up to reversing arrows 
(i.e., up to a sign modulo cycle relations),
we can assume that there are two edges as follows:
$$
\begin{tikzpicture}
\draw (1.8,1) circle [radius=0.07];
\draw[->] (1.9,1) -- (2.2,1);
\draw (2.3,1) circle [radius=0.07];
\node at (2.3,0.6) {$1$};
\draw[->] (2.4,1) -- (2.7,1);
\draw (2.8,1) circle [radius=0.07];
\end{tikzpicture}
$$
Then, modulo cycle relations, this is equivalent to 
$$
\begin{tikzpicture}
\draw (1.8,1) circle [radius=0.07];
\draw[->] (1.9,1) -- (2.2,1);
\draw (2.3,1) circle [radius=0.07];
\node at (2.3,0.6) {$1$};
\draw[->] (2.4,1) -- (2.7,1);
\draw (2.8,1) circle [radius=0.07];
\node at (3.4,1) {$\equiv-$};
\draw (4,1) circle [radius=0.07];
\draw (4.5,1) circle [radius=0.07];
\node at (4.5,0.6) {$1$};
\draw[->] (4.6,1) -- (4.9,1);
\draw (5,1) circle [radius=0.07];
\draw[<-] (4,1.1) to [out=90,in=90] (5,1.1);
\node at (5.5,1) {$-$};
\draw (6,1) circle [radius=0.07];
\draw[->] (6.1,1) -- (6.4,1);
\draw (6.5,1) circle [radius=0.07];
\node at (6.5,0.6) {$1$};
\draw (7,1) circle [radius=0.07];
\draw[<-] (6,1.1) to [out=90,in=90] (7,1.1);
\end{tikzpicture}
$$
Hence, $\Gamma$ is equivalent
to a linear combination of graphs in which there are $\ell-1$ edges in or out of the vertex $1$.
Proceeding by induction, we get the claim.

Next, suppose that $1$ is a leaf of $\Gamma$ connected with an edge to the vertex $i$
(if $1$ is an isolated vertex, let $i=2$).
Denote by $\Gamma'\in\mc G(n-1)$
the subgraph of $\Gamma$ obtained by deleting the vertex $1$ and any edge attached to it.
Notice that, under the natural embedding of $\mc G(n-1)$ into $\mc G(n)$, every cycle relation in $\mc G(n-1)$ corresponds to a cycle relation in $\mc G(n)$.
By induction on $n$, $\Gamma'$ is equivalent, modulo cycle relations,
to a disjoint union of lines, one of which starts at the vertex $i$
and the others satisfy the conditions \eqref{eq:Gcond}.
Then $\Gamma$ is also a disjoint union of lines,
one of which starts with $1$.
This completes the proof.
\end{proof}
\begin{remark}\label{rem:lines}
In fact, in Theorem \ref{lem:res5} below
we will prove that the set $\mc L(n)$ is a basis for $\mb F\mc G(n)/R(n)$,
where $R(n)$ is the subspace spanned by the cycle relations.
\end{remark}

\subsection{$\Gamma$-residues}\label{sec:4.2}

Given $i\neq j\in\{1,\dots,n\}$, 
we define the residue map
\begin{equation}\label{eq:res}
\Res_{z_j}\!dz_i\colon
\mc O^{\star}_n(z_1,\dots,z_n)\to\mc O^{\star}_{n-1}(z_1,\stackrel{i}{\check{\dots}},z_n)
\end{equation}
where $\stackrel{i}{\check{\dots}}$ means that the variable $z_i$ is skipped.
It is defined as the residue of a function $f(z_1,\dots,z_n)$,
viewed as a function of $z_i$, at $z_i=z_j$,
and is given by Cauchy's formula.
Explicitly, let
\begin{equation}\label{eq:fg}
f(z_1,\dots,z_n)
=
z_{ij}^{-\ell-1} g(z_1,\dots,z_n)
\,\in\mc O^{\star}_n
\,,
\end{equation}
where $\ell\in\mb Z$ and $g$ has neither a zero nor a pole at $z_i=z_j$.
Then
\begin{equation}\label{eq:cauchy}
\Res_{z_j}\!dz_i\, f(z_1,\dots,z_n)
=
\frac1{\ell!}\frac{\partial^\ell g}{\partial z_i^\ell}(z_1,\dots,{\smallunderbrace{z_j}_i},\dots,z_n)
\,\,\text{ if }\,\, \ell\geq0
\,,
\end{equation}
and it is zero for $\ell<0$.

Next, given a line $L=i_1\to i_2\to\dots\to i_k$, we define the map
\begin{equation}\label{eq:resL1}
\Res_{w} dL
\colon
\mc O^{\star}_n(z_1,\dots,z_n)
\to
\mc O^{\star}_{n-k+1}(z_1,\stackrel{i_1\dots i_k}{\check{\dots}},z_n,w)
\,,
\end{equation}
given by
\begin{equation}\label{eq:resL2}
\Res_{w} dL
\, f(z_1,\dots,z_n)
=
\Res_{z_{i_k}}\!dz_{i_{k-1}}
\cdots
\Res_{z_{i_2}}\!dz_{i_1}
\, f(z_1,\dots,z_n)
\,\Big|_{z_{i_k}=w}
\,.
\end{equation}
For example, if $L$ is a single vertex $i$ (i.e., $k=1$),
then the residue map \eqref{eq:resL1} is just the substitution $z_i=w$,
while if $L=i\to j$ is of length $2$,
then we recover the residue map \eqref{eq:res}:
$$
\Res_{z_j}\!dL
\, f(z_1,\dots,z_n)
=
\Res_{z_j}\!dz_i\, f(z_1,\dots,z_n)
\,.
$$

Finally, let $\Gamma\in\mc L(n)$ be a disjoint union of lines $L_1\sqcup\dots\sqcup L_p$
as in \eqref{eq:Gamma}.
In this case, we define the $\Gamma$-\emph{residue} map
\begin{equation}\label{eq:resG1}
\Res_{w_1,\dots,w_p}\! d\Gamma
\colon
\mc O^{\star}_n(z_1,\dots,z_n)
\to
\mc O^{\star}_{p}(w_1,\dots,w_p)
\,,
\end{equation}
given by
\begin{equation}\label{eq:resG2}
\Res_{w_1,\dots,w_p}\! d\Gamma
=
\Res_{w_1}\! dL_1
\circ\dots\circ
\Res_{w_p}\! dL_p
\,.
\end{equation}
Note that, by definition, we have
\begin{equation}\label{eq:resGmany}
\Res_{w_1,\dots,w_p}\! d\Gamma (z_i f)
=
w_\ell \Res_{w_1,\dots,w_p}\! d\Gamma f
\,\,\text{ if }\,\,
i=i^\ell_{k_\ell},\,\,1\leq\ell\leq p
\,.
\end{equation}
In the following lemmas we list some elementary properties of the $\Gamma$-residue maps,
which will be needed later.
\begin{lemma}\label{lem:res1}
For every\/ $\Gamma\in\mc L(n)$,
the\/ $\Gamma$-residue map \eqref{eq:resG1}
preserves the translation invariance of functions,
i.e.,
$$
\Res_{w_1,\dots,w_s}\! d\Gamma
\colon
\mc O^{\star,T}_n(z_1,\dots,z_n)
\to
\mc O^{\star,T}_p(w_1,\dots,w_p)
\,.
$$
\end{lemma}
\begin{proof}
It is enough to prove it for the map \eqref{eq:res},
in which case it is obvious.
\end{proof}
\begin{lemma}\label{lem:res2}
Let $\Gamma\in\mc L(n)$ be a graph as in \eqref{eq:Gamma};
in particular, $|E(\Gamma)|=n-p$.
Then
\begin{equation}\label{eq:ciao2}
\Res_{w_1,\dots,w_p}\! d\Gamma
\colon
\fil^r\mc O^{\star}_n(z_1,\dots,z_n)
\to
\fil^{r+p-n}\mc O^{\star}_p(w_1,\dots,w_p)
\,.
\end{equation}
In particular,
\begin{equation}\label{eq:ciao3}
\Res_{w_1,\dots,w_p}\! d\Gamma
(\fil^r\mc O^{\star}_n)
=0
\,\,\text{ for }\,\,
r<|E(\Gamma)|
\,.
\end{equation}
\end{lemma}
\begin{proof}
By the definition \eqref{eq:resL2}-\eqref{eq:resG2} of the $\Gamma$-residue map,
it is enough to prove that
$$
\Res_{z_j}\!dz_i
\colon
\fil^r\mc O^{\star}_n(z_1,\dots,z_n)
\to
\fil^{r-1}\mc O^{\star}_{n-1}(z_1,\stackrel{i}{\check{\dots}},z_n)
\,.
$$
This is immediate, by Cauchy's formula \eqref{eq:cauchy}.
Indeed, if $f\in\fil^r\mc O^{\star}_n$ is as in \eqref{eq:fg} with $\ell\geq0$,
then $g\in\fil^{r-1}\mc O^{\star}_n$,
and hence 
the right-hand side of \eqref{eq:cauchy} lies in $\fil^{r-1}\mc O^{\star}_{n-1}$.
By induction, we get \eqref{eq:ciao2}.
Equation \eqref{eq:ciao3} is an obvious consequence of \eqref{eq:ciao2}.
\end{proof}
\begin{lemma}\label{lem:res3}
Let $\Gamma\in\mc L(n)$ be as in \eqref{eq:Gamma}.
For a function $f\in\mc O^{\star}_n$,
and $i\in\{1,\dots,n\}$, we have
\begin{equation}\label{eq:ciao1}
\Res_{w_1,\dots,w_p}\! d\Gamma \,
(\partial_{z_i} f)
=
\begin{cases}
\displaystyle{
\partial_{w_\ell}\Res_{w_1,\dots,w_p}\! d\Gamma\, f
}
&\text{ if }\,\, i=i^\ell_{k_\ell}\,,\,\,1\leq\ell\leq p\,,\\
0&\text{ if }\,\, i\not\in\{i^1_{k_1},\dots,i^p_{k_p}\}\,.
\end{cases}
\end{equation}
\end{lemma}
\begin{proof}
If $f(z_1,\dots,z_n)\in\mc O^{\star}_n$ is as in \eqref{eq:fg},
then by Taylor expanding $g$, viewed as a rational function in $z_i$,
at $z_i=z_j$, we have
$$
f(z_1,\dots,z_n)
=
\sum_{m=-\ell-1}^\infty z_{ij}^m
f_m(z_1,\stackrel{i}{\check{\dots}},z_n)
\,,
\quad
f_m=\frac{1}{(m+\ell+1)!}\frac{\partial^{m+\ell+1}g}{\partial z_i^{m+\ell+1}}\Big|_{z_i=z_j}
\,.
$$
Then, by Cauchy's formula \eqref{eq:cauchy}, we have
\begin{equation}\label{eq:cauchy2}
\Res_{z_j}\!dz_i\/ f
=
f_{-1}
\,.
\end{equation}
It follows from \eqref{eq:cauchy2} that
\begin{equation}\label{eq:cauchy3}
\Res_{z_j}\!dz_i\/ \frac{\partial f}{\partial z_i}=0
\,,\,\text{ and }\,
\Res_{z_j}\!dz_i\/ \frac{\partial f}{\partial z_k}
=
\frac{\partial}{\partial z_k}
\Res_{z_j}\!dz_i\/ f
\,\,\,\,\text{ if } k\neq i
\,.
\end{equation}
Equation \eqref{eq:ciao1} is an immediate consequence of \eqref{eq:cauchy3}
and the definition \eqref{eq:resL2}-\eqref{eq:resG2} of the $\Gamma$-residue.
\end{proof}
\begin{proposition}\label{lem:res4}
Let\/ $\Gamma,\Gamma'\in\mc L(n)$
be such that\/ $|E(\Gamma')|=|E(\Gamma)|$.
Then, for every $q\in\mc O_n$, we have
$$
\Res_{w_1,\dots,w_p}\! d\Gamma\, 
p_{\Gamma'}(z_1,\dots,z_n)
q(z_1,\dots,z_n)
=
\delta_{\Gamma,\Gamma'}
q(z_1,\dots,z_n)\big|_{z_{i^a_b}=w_a\,\forall a,b}
\,.
$$
\end{proposition}
\begin{proof}
If $|E(\Gamma)|=|E(\Gamma')|=0$, the statement trivially holds.
Let $e=i\to j$ be the first edge of the first line of $\Gamma$ that is not a single vertex.
In other words, $e=1\to i^1_2$ if $k_1\geq2$, 
and, in general, $e=i^\ell_1\to i^\ell_2$ for the smallest $\ell$ such that $k_\ell\geq2$.

Observe that if neither $i\to j$ nor $j\to i$ is an edge of the graph $\Gamma'$,
then $p_{\Gamma'}$ has no pole at $z_i=z_j$, and hence
$$
\Res_{z_j}\!dz_i\, p_{\Gamma'}q=0
\,.
$$
If instead $e=i\to j$ is an edge of $\Gamma'$,
then
$$
p_{\Gamma'}=\frac1{z_{ij}}p_{\Gamma'\backslash e}
\,.
$$
Hence, by Cauchy's formula \eqref{eq:cauchy}, we have
$$
\Res_{z_j}\!dz_i\, p_{\Gamma'}q
=
(p_{\Gamma'\backslash e}\,q)\big|_{z_i=z_j}
=
p_{\bar\Gamma'}(z_1,\stackrel{i}{\check{\dots}},z_n)
\cdot
q|_{z_i=z_j}\,,
$$
where $\bar\Gamma'$ is the graph obtained from $\Gamma'$
by contracting the edge $e$ into a single vertex labeled $j$.
We then have
$$
\Res_{w_1,\dots,w_p}\!d\Gamma\, p_{\Gamma'}q
=
\Res_{w_1,\dots,w_p}\!d\bar\Gamma\, p_{\bar\Gamma'}q|_{z_i=z_j}
\,,
$$
where $\bar\Gamma$ is the graph obtained from $\Gamma$
by contracting the edge $e$ into a single vertex labeled $j$.

Note that both $\bar\Gamma$ and $\bar\Gamma'$ have the same number of edges
and lie in $\mc L(n-1)$ after the relabeling of the vertices 
$\phi\colon\{1,\stackrel{i}{\check{\dots}},n\}\to\{1,\dots,n-1\}$ given by
$$
\phi(m)=
\begin{cases}
m \,, &\text{ for }\,\, m<j\,,\\
i \,, &\text{ for }\,\, m=j\,,\\
m-1 \,, &\text{ for }\,\, m>j\,.\\
\end{cases}
$$
As a consequence, we get by induction that 
$$
\Res_{w_1,\dots,w_p}\!d\bar\Gamma\, p_{\bar\Gamma'}q|_{z_i=z_j}
=
\delta_{\bar\Gamma,\bar\Gamma'}
q(z_1,\dots,z_n)\big|_{z_{i^a_b}=w_a\,\forall a,b}
\,.
$$
If $\bar\Gamma\neq\bar\Gamma'$, then $\Gamma\neq\Gamma'$.
Conversely, if $\bar\Gamma=\bar\Gamma'$,
then $\Gamma=\Gamma'$ since they both lie in $\mc L(n)$
and $i<j$.
The claim follows.
\end{proof}
\begin{theorem}\label{lem:res5}
The set $\mc L(n)$ is a basis for the quotient space $\mb F\mc G(n)/R(n)$,
where $\mb F\mc G(n)$ is the vector space with basis the set of graphs $\mc G(n)$,
and $R(n)$ is the subspace spanned by the cycle relations (i) and (ii) from Section \ref{sec:4.1}.
\end{theorem}
\begin{proof}
We already know by Lemma \ref{prop:basis}
that $\mc L(n)$ spans $\mb F\mc G(n)$ modulo cycle relations.
Hence, we only need to prove linear independence.
Let
$$
\sum_{\Gamma\in\mc L(n)}c_\Gamma \Gamma
\in
R(n)
\,,\qquad
c_\Gamma\in\mb F
\,.
$$
Since the cycle relations are homogeneous in the number of edges,
we can assume that all the graphs $\Gamma$ appearing above have the
same number of edges, $s$.
Then,
$\sum_{\Gamma\in\mc L(n)}c_\Gamma \Gamma$
is a linear combination of graphs $\Gamma_1\in\mc G(n)$
with $s$ edges and not acyclic,
and of
$\sum_{e\in C}\Gamma_2\backslash e$,
where $\Gamma_2\in\mc G(n)$ has $s+1$ edges and contains a cycle $C$.
It follows from Lemma \ref{lem:pgamma}(b) that
\begin{equation}\label{eq:ciao4}
\sum_{\Gamma\in\mc L(n)}c_\Gamma p_\Gamma
\end{equation}
is a linear combination of $p_{\Gamma_1}$,
where $\Gamma_1\in\mc G(n)$
have $s$ edges and are not acyclic.
Let $\Gamma\in\mc L(n)$ be as in \eqref{eq:Gamma} with $p=n-s$.
Applying $\Res_{w_1,\dots,w_p}\!d\Gamma$
to \eqref{eq:ciao4}, we get $c_\Gamma$ by Proposition \ref{lem:res4}.
On the other hand, by Lemma \ref{lem:pgamma}(a) and equation \eqref{eq:ciao3},
we get $c_\Gamma=0$.
\end{proof}

\subsection{$\Gamma$-Fourier transforms}\label{sec:4.4}

Let $\Gamma=L_1\sqcup\dots\sqcup L_p\in\mc L(n)$ be a disjoint
union of lines as in \eqref{eq:Gamma}.
We define the $\Gamma$-\emph{exponential} function as
\begin{equation}\label{eq:gammaexp1}
E^{\Gamma}_{\lambda_1,\dots,\lambda_n}(z_1,\dots,z_n)
=
\prod_{\ell=1}^p 
E^{L_\ell}_{\lambda_{i^\ell_1},\dots,\lambda_{i^\ell_{k_\ell}}}
(z_{i^\ell_1},\dots,z_{i^\ell_{k_\ell}})
\,\in\mc O^{T}_n[[\lambda_1,\dots,\lambda_n]]
\,,
\end{equation}
where, for a line $L=i_1\to i_2\to\dots\to i_k$, we let
\begin{equation}\label{eq:gammaexp2}
E^{L}_{\lambda_{i_1},\dots,\lambda_{i_k}}(z_{i_1},\dots,z_{i_k})
=
\exp\Big(-\sum_{a=1}^{k-1}z_{i_ai_k}\lambda_{i_a}\Big)
\,.
\end{equation}
For example, if $k=1$ and $L$ consists of the single vertex $i$,
then 
$E^{L}_{\lambda_i}(z_i)=1$.
If $k=2$ and $L$ has one edge $i\to j$,
then 
$E^{L}_{\lambda_i,\lambda_j}(z_i,z_j)=e^{-z_{ij}\lambda_i}$.
\begin{definition}\label{def:Gfourier}
For $\Gamma\in\mc L(n,p)$ and $f\in\mc O^{\star}_n$,
the $\Gamma$-\emph{Fourier transform} of $f$ is
\begin{equation}\label{eq:Gfourier}
\mc F^\Gamma_{\lambda_1,\dots,\lambda_n}
(f;w_1,\dots,w_p)
=
\Res_{w_1,\dots,w_p}\!d\Gamma\,
f(z_1,\dots,z_n)E^{\Gamma}_{\lambda_1,\dots,\lambda_n}(z_1,\dots,z_n)
\,.
\end{equation}
If we do not need to specify the variables $w_1,\dots,w_p$, we will use the simplified notation
$\mc F^\Gamma_{\lambda_1,\dots,\lambda_n}(f)$.
\end{definition}
\begin{example}\label{ex:fourier1}
If $p=n$ and $\Gamma=\bullet\cdots\bullet$ has no edges,
the corresponding Fourier transform is
$$
\mc F^{\bullet\cdots\bullet}_{\lambda_1,\dots,\lambda_n}
(f;w_1,\dots,w_n)
=
f(w_1,\dots,w_n)
\,.
$$
\end{example}
\begin{example}\label{ex:fourier2}
If $p=1$ and $\Gamma=L=1\to\cdots\to n$ is a single line,
the corresponding Fourier transform is
$$
\displaystyle{
\mc F^{1\to\cdots\to n}_{\lambda_1,\dots,\lambda_n}
(f;w)
=
\Res_{w}\!dz_{n-1}
\cdots
\Res_{z_2}\!dz_1
e^{-\sum_{i=1}^{n-1}(z_i-w)\lambda_{i}}
\, f(z_1,\dots,z_{n-1},w)
\,.}
$$
Note that, by Lemma \ref{lem:Gfourier1} below, if $f\in\mc O^{\star T}_n$,
then 
$$
\mc F^{1\to\cdots\to n}_{\lambda_1,\dots,\lambda_n}(f;w)
=
\mc F^{1\to\cdots\to n}_{\lambda_1,\dots,\lambda_n}(f;0)
\in
\mb F[\lambda_1,\dots,\lambda_n]
\,
$$
since $\mc O^{\star T}_1=\mb F$. In particular, the Fourier transform is independent of $w$.
\end{example}
\begin{lemma}\label{lem:Gfourier1}
Let\/ $\Gamma\in\mc L(n)$ be as in \eqref{eq:Gamma} and\/ $f\in\mc O^{\star}_n$.
Then\/
$\mc F^\Gamma_{\lambda_1,\dots,\lambda_n}
(f)\in\mc O^{\star}_p[\lambda_1,\dots,\lambda_n]$.
If\/ $f\in\mc O^{\star T}_n$, then\/
$\mc F^\Gamma_{\lambda_1,\dots,\lambda_n}
(f)\in\mc O^{\star T}_p[\lambda_1,\dots,\lambda_n]$.
\end{lemma}
\begin{proof}
For the first claim, note that the $\Gamma$-residue and the $\Gamma$-exponential
are defined as products over the lines $L_1,\dots,L_p$
in $\Gamma$.
Hence, it is enough to consider the case
of a single line $L=i_1\to i_2\to\dots \to i_k$.
We have
$$
\sum_{a=1}^{k-1}z_{i_ai_k}\lambda_{i_a}
=
z_{i_1i_2}\lambda_{i_1}
+
\sum_{a=2}^{k-1}z_{i_ai_k}\widetilde{\lambda}_{i_a}
\,,
$$
where $\widetilde{\lambda}_{i_a}=\lambda_{i_a}+\delta_{a,2}\lambda_{i_1}$.
Expanding the exponential $e^{-z_{i_1i_2}\lambda_{i_1}}$,
since $f$ has a pole at $z_{i_1i_2}$ of finite order,
we obtain that
$$
\Res_{z_{i_2}}\!dz_{i_1}\,
f(z_1,\dots,z_n) 
E^{L}_{\lambda_{i_1},\dots,\lambda_{i_k}}(z_{i_1},\dots,z_{i_k})
$$
is a polynomial in $\lambda_{i_1}$.
Proceeding by induction, we get the first claim.
The second claim follows from Lemma \ref{lem:res1} and the fact that the $\Gamma$-exponential
\eqref{eq:gammaexp1} is translation invariant.
\end{proof}
\begin{lemma}\label{lem:Gfourier2}
Let\/ $\Gamma\in\mc L(n)$ with\/ $|E(\Gamma)|=s$.
Then\/
$
\mc F^\Gamma_{\lambda_1,\dots,\lambda_n}
(\fil^r\mc O^{\star}_n)=0
$
for all\/ $r<s$.
\end{lemma}
\begin{proof}
It follows immediately from \eqref{eq:ciao3}.
\end{proof}
\begin{lemma}\label{lem:Gfourier3}
Let\/ $\Gamma,\Gamma'\in\mc L(n)$ with\/ $|E(\Gamma)|=|E(\Gamma')|$.
Then\/
$
\mc F^\Gamma_{\lambda_1,\dots,\lambda_n}
(p_{\Gamma'})=\delta_{\Gamma,\Gamma'}
$.
\end{lemma}
\begin{proof}
It is an obvious consequence of Proposition \ref{lem:res4}.
\end{proof}
\begin{lemma}\label{lem:Gfourier4}
For $\Gamma\in\mc L(n,p)$ as in \eqref{eq:Gamma}
and $f\in\mc O^{\star}_n$, we have
\begin{equation}\label{eq:bye1}
\mc F^\Gamma_{\lambda_1,\dots,\lambda_n}
(\partial_{z_i} f)
=
\begin{cases}
\displaystyle{
\Big(\partial_{w_\ell}-\sum_{a=1}^{k_\ell-1}\lambda_{i^\ell_a}\Big)
\mc F^\Gamma_{\lambda_1,\dots,\lambda_n}(f)
}
&\text{ if }\,\, i=i^\ell_{k_\ell}\,,\,\,1\leq\ell\leq p\,,\\
\lambda_i\,
\mc F^\Gamma_{\lambda_1,\dots,\lambda_n}(f)
&\text{ if }\,\, i\not\in\{i^1_{k_1},\dots,i^p_{k_p}\}\,.
\end{cases}
\end{equation}
\end{lemma}
\begin{proof}
It follows from Lemma \ref{lem:res3}
and the definition \eqref{eq:gammaexp1}--\eqref{eq:gammaexp2}
of $\Gamma$-exponential functions.
\end{proof}
\begin{lemma}\label{lem:Gfourier5}
For $\Gamma\in\mc L(n,p)$ as in \eqref{eq:Gamma}
and $f\in\mc O^{\star}_n$, we have
\begin{equation}\label{eq:bye2}
\mc F^\Gamma_{\lambda_1,\dots,\lambda_n}
(z_i f)
=
(w_\ell-\partial_{\lambda_i})
\mc F^\Gamma_{\lambda_1,\dots,\lambda_n}(f)
\,,
\end{equation}
for
$i\in\{i^\ell_1,\dots,i^\ell_{k_\ell}\}$, $1\leq \ell\leq p$.
Note that $\partial_{\lambda_i}\mc F^\Gamma_{\lambda_1,\dots,\lambda_n}(f)=0$
for $i=i^\ell_{k_\ell}$.
\end{lemma}
\begin{proof}
It follows from the definition \eqref{eq:gammaexp1}--\eqref{eq:gammaexp2}
of the $\Gamma$-exponential and equation \eqref{eq:resGmany}.
\end{proof}

\subsection{Convolution product}\label{sec:4.4b}

In this section, we define a bilinear \emph{convolution product}
\begin{equation}\label{eq:notaction2}
\mc O^{\star}_p \times \mb F[\Lambda_1,\dots,\Lambda_p] 
\to \mb F[\Lambda_1,\dots,\Lambda_p] 
\,,\qquad
(F,Q) \mapsto F*Q(\Lambda_1,\dots,\Lambda_p) 
\,,
\end{equation}
as follows.
First, we introduce the map
\begin{equation}\label{eq:iota}
\iota_{w_1,\dots,w_p}
\colon
\mc O^{\star}_p(w_1,\dots,w_p)
\to
\mb F((w_1))\cdots((w_{p-1}))[[w_p]]
\,,
\end{equation}
defined as the geometric series expansion in the domain
$|w_1|>|w_2|>\dots>|w_p|$.
Then, for $F\in\mc O^{\star}_p$ and 
$Q\in\mb F[\Lambda_1,\dots,\Lambda_p]$,
we let
\begin{equation}\label{eq:notaction3}
F*Q(\Lambda_1,\dots,\Lambda_p) =
\iota_{w_1,\dots,w_p} F(w_1,\dots,w_p)
\big|_{w_\ell=-\partial_{\Lambda_\ell}}
Q(\Lambda_1,\dots,\Lambda_p) 
\,.
\end{equation}

Let us explain why the right-hand side of \eqref{eq:notaction3}
is a well-defined polynomial.
We expand the formal Laurent series 
$\widetilde F=\iota_{w_1,\dots,w_p}F$
and the polynomial $Q$ as
$$
\widetilde F
=
\sum_{a\in\mb Z^p} c_{a_1,\dots,a_p}
w_1^{a_1}\cdots w_p^{a_p}
\;\text{ and }\;
Q
=
\sum_{b\in\mb Z_+^p} d_{b_1,\dots,b_p}
\Lambda_1^{(b_1)}\cdots\Lambda_p^{(b_p)}
\,.
$$
Here and below we use the divided power notation 
$\Lambda_\ell^{(b_\ell)}:=\Lambda_\ell^{b_\ell} / b_\ell!$ for $b_\ell\geq0$ and $\Lambda_\ell^{(b_\ell)}=0$ for $b_\ell<0$.
Then we have
\begin{equation}\label{eq:notaction}
\begin{split}
\widetilde F(w_1,&\dots,w_p)
\big|_{w_\ell=-\partial_{\Lambda_\ell}}
Q(\Lambda_1,\dots,\Lambda_p) \\
&=
\sum_{a,b} c_{a_1,\dots,a_p}d_{b_1,\dots,b_p}
(-1)^{a_1+\cdots+a_p}
\Lambda_1^{(b_1-a_1)}\cdots\Lambda_p^{(b_p-a_p)}
\,.
\end{split}
\end{equation}
We claim that the sum in the right-hand side of \eqref{eq:notaction} is finite.
Indeed, 
the sum over $b\in\mb Z_+^p$ is finite.
For every fixed $b$,
the sum over $a_p$ is finite;
for every $a_p$ the sum over $a_{p-1}$ is finite, and so on.
Hence, \eqref{eq:notaction} is a well-defined polynomial.

Note that \eqref{eq:iota} is an algebra homomorphism.
However, the convolution product \eqref{eq:notaction2}
does not define an action of the algebra $\mc O^{\star}_p$
on the space of polynomials. 
For example, we have
$$
\frac1{w_1-w_2}*\big((w_1-w_2)*1\big)=0
\,\,\text{ while }\,\,
\frac{w_1-w_2}{w_1-w_2}*1=1
\,.
$$
Nevertheless, we have the following lemma.
\begin{lemma}\label{lem:notaction}
For every\/ $F\in\mc O^{\star}_p(w_1,\dots,w_p)$ and\/ $Q\in\mb F[\Lambda_1,\dots,\Lambda_p]$, we have $(\ell=1,\dots,p):$
\begin{align}
& (w_\ell F)*Q
=
-\partial_{\Lambda_\ell}(F*Q)
\,,\label{eq:notaction4}\\
& \Lambda_\ell (F*Q) - F*(\Lambda_\ell Q) 
= 
(\partial_{w_\ell}F)*Q
\,.\label{eq:notaction5}
\end{align}
\end{lemma}
\begin{proof}
For simplicity of notation, let $\ell=1$. By linearity, we can assume that 
$F=w_1^{a_1}\cdots w_p^{a_p}$ and $Q=\Lambda_1^{(b_1)}\cdots\Lambda_p^{(b_p)}$,
where $a_i\in\mb Z$, $b_i\in\mb Z_+$. Then using \eqref{eq:notaction}, we find:
\begin{align*}
(w_1 F)*Q &=
(-1)^{1+a_1+\cdots+a_p} \,
\Lambda_1^{(b_1-a_1-1)} \Lambda_2^{(b_2-a_2)}\cdots\Lambda_p^{(b_p-a_p)}
\,, \\
\partial_{\Lambda_1}(F*Q) &=
(-1)^{a_1+\cdots+a_p} \,
\Lambda_1^{(b_1-a_1-1)} \Lambda_2^{(b_2-a_2)}\cdots\Lambda_p^{(b_p-a_p)}
\,, \\
\Lambda_1 (F*Q) &= 
(-1)^{a_1+\cdots+a_p} \, (b_1+1-a_1)
\Lambda_1^{(b_1+1-a_1)} \Lambda_2^{(b_2-a_2)}\cdots\Lambda_p^{(b_p-a_p)}
\,, \\
F*(\Lambda_1 Q) &= 
(-1)^{a_1+\cdots+a_p} \, (b_1+1)
\Lambda_1^{(b_1+1-a_1)} \Lambda_2^{(b_2-a_2)}\cdots\Lambda_p^{(b_p-a_p)}
\,, \\
(\partial_{w_\ell}F)*Q &=
(-1)^{1+a_1+\cdots+a_p} \, a_1
\Lambda_1^{(b_1+1-a_1)} \Lambda_2^{(b_2-a_2)}\cdots\Lambda_p^{(b_p-a_p)}
\,.
\end{align*}
The claim follows.
\end{proof}
Note that neither side of \eqref{eq:notaction4}
is necessarily equal to $-F*(\partial_{\Lambda_\ell}Q)$.
Indeed, in the same setting as in the proof of Lemma \ref{lem:notaction},
we have
$$
F*(\partial_{\Lambda_1}Q)
=
(-1)^{a_1+\cdots+a_p} \,
\Lambda_1^{(b_1-a_1-1)} \Lambda_2^{(b_2-a_2)}\cdots\Lambda_p^{(b_p-a_p)}
\,,
$$
unless $a_1<b_1=0$, in which case $F*(\partial_{\Lambda_1}Q)=0$,
while the right-hand side above is not $0$.

\section{Relation between chiral and classical operads}\label{sec:6}

\subsection{Preliminary notation}\label{sec:5.1}

Let $V$ be a superspace with an even endomorphism $\partial$,
which is $\mb Z_+$-graded by $\mb F[\partial]$-submodules:
\begin{equation}\label{eq:Vgrading}
V=\bigoplus_{s\in\mb Z_+}V_s\,.
\end{equation}
We have the induced increasing filtration by $\mb F[\partial]$-submodules
\begin{equation}\label{eq:Vfil}
\fil^t V=\bigoplus_{s=0}^tV_s\,,
\end{equation}
and the associated graded $\gr V$ is canonically isomorphic to $V$.

Let $Y\in \gr^r P^{\cl}(n)$ and let $\Gamma\in\mc L(n,p)$
be as in \eqref{eq:Gamma}.
Recall that, for $v\in V^{\otimes n}$, 
by the sesquilinearity axiom \eqref{eq:sesq1},
$Y^\Gamma_{\lambda_1,\dots,\lambda_n}(v)$ is a polynomial in the variables
\begin{equation}\label{eq:Lambdas}
\Lambda_\ell
=
\lambda_{L_\ell}
=
\lambda_{i^\ell_1}+\dots+\lambda_{i^\ell_{k_\ell}}
\,,\qquad
\ell=1,\dots,p
\,.
\end{equation}
By an abuse of notation, we shall then alternatively write
$$
Y^\Gamma_{\Lambda_1,\dots,\Lambda_p}(v)
=
Y^\Gamma_{\lambda_1,\dots,\lambda_n}(v)
\,,\qquad
\Gamma\in\mc L(n,p)
\,.
$$

Recall also that, for $i=1,\dots,n$, $\partial_i\colon V^{\otimes n}\to V^{\otimes n}$
denotes the action of $\partial$ on the $i$-th factor.
For a polynomial 
$$
P(x_1,\dots,x_n)
=
\sum c_{j_1,\dots,j_n}
x_1^{j_1}\cdots x_n^{j_n}
\,,
$$
we will use the notation
\begin{equation}\label{eq:notation}
P(x_1,\dots,x_n)\big(\big|_{x_i=\partial_i}\,v\big)
=
\sum c_{j_1,\dots,j_n}
\partial_1^{j_1}\cdots \partial_n^{j_n}v
\,.
\end{equation}

\subsection{The inverse map}\label{sec:5.2}

Given $f\in\mc O^{\star T}_n$
and $v\in V^{\otimes n}$,
we define
\begin{equation}\label{eq:formula}
X_{\lambda_1,\dots,\lambda_n}(v\otimes f)
=
\sum_{p=1}^n
\sum_{\Gamma\in\mc L(n,p)}
\!\!\!
\mc F^\Gamma_{\lambda_1+x_1,\dots,\lambda_n+x_n}(f) *
Y^\Gamma_{\Lambda_1,\dots,\Lambda_p}\big(\big|_{x_i=\partial_i}v\big)
\,.
\end{equation}
Let us explain the meaning of this formula.
By Lemma \ref{lem:Gfourier1},
the $\Gamma$-Fourier transform is a finite sum
$$
\mc F^\Gamma_{\lambda_1+x_1,\dots,\lambda_n+x_n}(f)
=
\sum_{a\in\mb Z_+^n}
F_a(w_1,\dots,w_p)\,(\lambda_1+x_1)^{(a_1)}\cdots(\lambda_n+x_n)^{(a_n)}
$$
with coefficients $F_a\in\mc O^{\star T}_p$.
According to the notation \eqref{eq:notation},
we apply the $x_i$ as $\partial_i$ on the vector $v$
in the argument of $Y^\Gamma_{\Lambda_1,\dots,\Lambda_p}$.
Then we take the convolution product \eqref{eq:notaction2} 
of each $F_a$ 
with $Y^\Gamma_{\Lambda_1,\dots,\Lambda_p}$,
which is a polynomial in $\Lambda_1,\dots,\Lambda_p$.
As a result, each summand in the right-hand side of \eqref{eq:formula} is
\begin{equation}\label{eq:formula2}
\sum_{a,b\in\mb Z_+^n}\lambda_1^{(a_1)}\cdots\lambda_n^{(a_n)}
F_{a+b}*
Y^\Gamma_{\Lambda_1,\dots,\Lambda_p}
\big(\partial_1^{(b_1)}\cdots\partial_n^{(b_n)}v\big)
\,.
\end{equation}
Finally, we make the substitution \eqref{eq:Lambdas}
to get a polynomial in $\lambda_1,\dots,\lambda_n$
with coefficients in $V$.
\begin{theorem}\label{thm:main}
Let $V$ be a superspace with an even endomorphism $\partial$,
endowed with a $\mb Z_+$-grading \eqref{eq:Vgrading} by $\mb F[\partial]$-submodules,
and with the associated increasing filtration \eqref{eq:Vfil}.
Then for $Y\in\gr^r P^{\cl}(V)(n)$, formula \eqref{eq:formula} 
defines an element $X\in\fil^r P^{\ch}(V)(n)$.
The obtained linear map
$\gr^r P^{\cl}(V)(n)\to\fil^r P^{\ch}(V)(n)$
sending $Y\mapsto X$
induces a map $\gr^r P^{\cl}(V)(n)\to\gr^r P^{\ch}(V)(n)$, which is
the inverse of the map \eqref{eq:map}.
Consequently, 
the homomorphism of operads $\gr P^{\ch}(V)(n)\to \gr P^{\cl}(\gr V)(n)$
defined by \eqref{eq:map} is an isomorphism.
\end{theorem}
\begin{proof}


First, we prove that for $f\in\mc O^{\star T}_n$ the right-hand side of \eqref{eq:formula}
is a well-defined element of the quotient space
$V[\lambda_1,\dots,\lambda_n]/\langle\partial+\lambda_1+\dots+\lambda_n\rangle$,
i.e., it does not depend on the choice of the representative
of $Y^\Gamma$ in the quotient space
$V[\Lambda_1,\dots,\Lambda_p]/\langle\partial+\Lambda_1+\dots+\Lambda_p\rangle$.
Indeed, suppose
$$
Y^\Gamma_{\Lambda_1,\dots,\Lambda_p}
\big(\partial_1^{(b_1)}\cdots\partial_n^{(b_n)}v\big)
=
(\partial+\Lambda_1+\dots+\Lambda_p)Q_b
\in\langle\partial+\Lambda_1+\dots+\Lambda_p\rangle
\,,
$$
for some $Q_b\in V[\Lambda_1,\dots,\Lambda_p]$.
Then, by Lemma \ref{lem:notaction}, we have
\begin{align*}
F&_{a+b}*
Y^\Gamma_{\Lambda_1,\dots,\Lambda_p}
\big(\partial_1^{(b_1)}\cdots\partial_n^{(b_n)}v\big)
\\
& =
F_{a+b}*
\big((\partial+\Lambda_1+\dots+\Lambda_p)Q_b\big)
\\
&=
(\partial+\Lambda_1+\dots+\Lambda_p)
(F_{a+b}*Q_b)
-
\big((\partial_{w_1}+\dots+\partial_{w_p})F_{a+b}\big)*Q_b
\\
&=
(\partial+\lambda_1+\dots+\lambda_n)
(F_{a+b}*Q_b)
\equiv0
\,,
\end{align*}
since $F_{a+b}$ is translation invariant.


Next, we check that the map 
$$
X\colon V^{\otimes n}\otimes\mc O^{\star T}_n
\to
V[\lambda_1,\dots,\lambda_n]/\langle\partial+\lambda_1+\dots+\lambda_n\rangle
$$
satisfies the sesquilinearity relations \eqref{20160629:eq4a}--\eqref{20160629:eq4b},
i.e., it is a chiral map $X\in P^{\ch}(n)$.
For $i\not\in\{i^1_{k_1},\dots,i^p_{k_p}\}$, we have,
by Lemma \ref{lem:Gfourier4},
\begin{align*}
X&_{\lambda_1,\dots,\lambda_n}(v\otimes \partial_{z_i}f)
\\
&=
\sum_{p=1}^n
\sum_{\Gamma\in\mc L(n,p)}
\!\!\!
\mc F^\Gamma_{\lambda_1+x_1,\dots,\lambda_n+x_n}(\partial_{z_i}f) *
Y^\Gamma_{\Lambda_1,\dots,\Lambda_p}\big(\big|_{x_i=\partial_i}v\big)
\\
&
=
\sum_{p=1}^n
\sum_{\Gamma\in\mc L(n,p)}
\!\!\!
(\lambda_i+x_i) \mc F^\Gamma_{\lambda_1+x_1,\dots,\lambda_n+x_n}(f) *
Y^\Gamma_{\Lambda_1,\dots,\Lambda_p}\big(\big|_{x_i=\partial_i}v\big)
\\
&
=X_{\lambda_1,\dots,\lambda_n}((\lambda_i+\partial_i)v\otimes f)
\,.
\end{align*}
Next, let $i=i^\ell_{k_\ell}$, $1\leq\ell\leq p$.
%
By Lemmas \ref{lem:Gfourier4} and \ref{lem:notaction}, we have
\begin{align*}
& \mc F^\Gamma_{\lambda_1+x_1,\dots,\lambda_n+x_n}(\partial_{z_i}f) *
Y^\Gamma_{\Lambda_1,\dots,\Lambda_p}\big(\big|_{x_i=\partial_i}v\big)
\\
& =
\bigg(
\Big(\partial_{w_\ell}-\sum_{a=1}^{k_\ell-1}(\lambda_{i^\ell_a}+x_{i^\ell_a})\Big)
\mc F^\Gamma_{\lambda_1+x_1,\dots,\lambda_n+x_n}(f) 
\bigg)
*
Y^\Gamma_{\Lambda_1,\dots,\Lambda_p}\big(\big|_{x_i=\partial_i}v\big)
\\
& =
\Big(\Lambda_\ell-\sum_{a=1}^{k_\ell-1}\lambda_{i^\ell_a}\Big)
\mc F^\Gamma_{\lambda_1+x_1,\dots,\lambda_n+x_n}(f) *
Y^\Gamma_{\Lambda_1,\dots,\Lambda_p}\big(\big|_{x_i=\partial_i}v\big)
\\
&-
\mc F^\Gamma_{\lambda_1+x_1,\dots,\lambda_n+x_n}(f) *
\bigg(
\Big(\Lambda_\ell+\sum_{a=1}^{k_\ell-1}x_{i^\ell_a}\Big)
Y^\Gamma_{\Lambda_1,\dots,\Lambda_p}\big(\big|_{x_i=\partial_i}v\big)
\bigg)
\\
& =
\lambda_i
\mc F^\Gamma_{\lambda_1+x_1,\dots,\lambda_n+x_n}(f) *
Y^\Gamma_{\Lambda_1,\dots,\Lambda_p}\big(\big|_{x_i=\partial_i}v\big)
\\
&+
\mc F^\Gamma_{\lambda_1+x_1,\dots,\lambda_n+x_n}(f) *
Y^\Gamma_{\Lambda_1,\dots,\Lambda_p}\big(\big|_{x_i=\partial_i}\partial_i v\big)
\,.
\end{align*}
For the last equality, we used the sesquilinearity \eqref{eq:sesq2} of $Y^\Gamma$.
This proves \eqref{20160629:eq4a}.

Next, let us prove equation \eqref{20160629:eq4b}.
Let $i\in\{i^\ell_1,\dots,i^\ell_{k_\ell}\}$ and $j\in\{i^h_1,\dots,i^h_{k_h}\}$.
By Lemma \eqref{lem:Gfourier5} 
and equation \eqref{eq:notaction4}, we have
\begin{align*}
& \mc F^\Gamma_{\lambda_1+x_1,\dots,\lambda_n+x_n}(z_{ij}f) *
Y^\Gamma_{\Lambda_1,\dots,\Lambda_p}\big(\big|_{x_i=\partial_i}v\big)
\\
& =
\big((w_\ell-w_h-\partial_{\lambda_i}+\partial_{\lambda_j})
\mc F^\Gamma_{\lambda_1+x_1,\dots,\lambda_n+x_n}(f)\big) *
Y^\Gamma_{\Lambda_1,\dots,\Lambda_p}\big(\big|_{x_i=\partial_i}v\big)
\\
& =
(-\partial_{\Lambda_\ell}+\partial_{\Lambda_h})\big(
\mc F^\Gamma_{\lambda_1+x_1,\dots,\lambda_n+x_n}(f) *
Y^\Gamma_{\Lambda_1,\dots,\Lambda_p}\big(\big|_{x_i=\partial_i}v\big)
\big)
\\
&\quad +
\big((-\partial_{\lambda_i}+\partial_{\lambda_j})
\mc F^\Gamma_{\lambda_1+x_1,\dots,\lambda_n+x_n}(f)\big) *
Y^\Gamma_{\Lambda_1,\dots,\Lambda_p}\big(\big|_{x_i=\partial_i}v\big)
\\
& =
(-\partial_{\lambda_i}+\partial_{\lambda_j})\big(
\mc F^\Gamma_{\lambda_1+x_1,\dots,\lambda_n+x_n}(f) *
Y^\Gamma_{\Lambda_1,\dots,\Lambda_p}\big(\big|_{x_i=\partial_i}v\big)
\big)
\,.
\end{align*}
For the last equality we used the chain rule 
and $\partial_{\lambda_i}\Lambda_{\ell'}=\delta_{\ell,\ell'}$.
This proves \eqref{20160629:eq4b}.
Hence, $X\in P^{\ch}(n)$.


Next, we prove that $X$ lies in the $r$-th filtered space $\fil^r P^{\ch}(n)$,
provided that $Y\in\gr^r P^{\cl}(n)$.
Let $f\in\fil^s\mc O^{\star T}_n$
and $v\in\fil^t(V^{\otimes n})$.
By Lemma \ref{lem:Gfourier2},
$\mc F^\Gamma_{\lambda_1,\dots,\lambda_n}(f)=0$
unless $|E(\Gamma)|=n-p\leq s$.
In this case, by the definition \eqref{pclgrading} of the grading of $P^{\cl}$,
we have
\begin{align*}
Y^\Gamma_{\Lambda_1,\dots,\Lambda_p}(v)
& \in
(\gr^{n-p+t-r}V)[\lambda_1,\dots,\lambda_n]/\langle\partial+\lambda_1+\dots+\lambda_n\rangle
\\
& \subset
(\fil^{s+t-r}V)[\lambda_1,\dots,\lambda_n]/\langle\partial+\lambda_1+\dots+\lambda_n\rangle
\,.
\end{align*}
The claim follows from the facts that the filtration of $V$ is invariant under the action of $\partial$
and the convolution product does not act on the coefficients (in $V$) of the polynomials.


To complete the proof of the theorem,
we are left to check that, for $X$ as in \eqref{eq:formula},
the image of $X$ under the map \eqref{eq:map} coincides with $Y$.
Indeed, by definition \eqref{gr3}, the image of $X$ under \eqref{eq:map}
maps $\Gamma'\in\mc G(n)$ with $s$ edges
and $\bar v=v+\fil^{t-1} V^{\otimes n}\in\gr^t V^{\otimes n}$
to
\begin{equation}
\label{eq:proof}
\begin{split}
\overline{
X_{\lambda_1,\dots,\lambda_n}(v\otimes p_{\Gamma'})
}
& =
\sum_{p=1}^n
\sum_{\Gamma\in\mc L(n,p)}
\!\!\!
\mc F^\Gamma_{\lambda_1+x_1,\dots,\lambda_n+x_n}(p_{\Gamma'}) *
Y^\Gamma_{\Lambda_1,\dots,\Lambda_p}\big(\big|_{x_i=\partial_i}v\big)
\\
& +
(\fil^{s+t-r-1}V)[\lambda_1,\dots,\lambda_n]/\langle\partial+\lambda_1+\dots+\lambda_n\rangle
\,.
\end{split}
\end{equation}
Since $p_{\Gamma'}\in\fil^{s}\mc O^{\star T}_n$,
by Lemma \ref{lem:Gfourier2} the sum over $p$ in \eqref{eq:proof} 
can be restricted by the inequality $n-p=|E(\Gamma)|\leq s=|E(\Gamma')|$.
On the other hand, if $n-p<s$, by \eqref{pclgrading}
we have 
$$
Y^\Gamma_{\Lambda_1,\dots,\Lambda_p}(v)
\in
(\fil^{s+t-r-1}V)[\lambda_1,\dots,\lambda_n]/\langle\partial+\lambda_1+\dots+\lambda_n\rangle
\,;
$$
hence the corresponding terms vanish in \eqref{eq:proof}.
We can thus restrict the sum over $p$ in \eqref{eq:proof} to $|E(\Gamma)|=|E(\Gamma')|$.
In this case, by Lemma \ref{lem:Gfourier3}, we have
$$
\mc F^\Gamma_{\lambda_1,\dots,\lambda_n}(p_{\Gamma'})
=
\delta_{\Gamma,\Gamma'}
\,.
$$
Therefore, the right-hand side of \eqref{eq:proof}
becomes $\overline{Y^\Gamma_{\lambda_1,\dots,\lambda_n}(v)}$,
completing the proof.
\end{proof}

\subsection{Examples}\label{sec:5.3}

In this section, we write down explicitly formula \eqref{eq:formula}
in some special cases.
%
%
First, assume that $Y\in P^{\cl}(V)(n)$ is such that
$Y^{\Gamma}=0$ unless $|E(\Gamma)|=0$.
For example, this happens for $Y\in\gr^0 P^{\cl}(V)(n)$, 
provided that $V=V_0$ has trivial grading \eqref{eq:Vgrading}.
Under the above assumption, only the summand with $p=n$
and the $n$-graph with no edges $\Gamma=\bullet\cdots\bullet$ 
is non-vanishing in \eqref{eq:formula}.
Thus, by Example \ref{ex:fourier1}, we get
\begin{equation}\label{eq:formula3}
X_{\lambda_1,\dots,\lambda_n}(v\otimes f)
=
f(w_1,\dots,w_n) * Y^{\bullet\cdots\bullet}_{\lambda_1,\dots,\lambda_n}(v)
\,.
\end{equation}


Next, assume that $Y\in P^{\cl}(n)$ is such that,
for $\Gamma\in\mc L(n)$,
$Y^{\Gamma}=0$ unless $\Gamma$ is the single line $1\to\cdots\to n$.
In this case, by Example \ref{ex:fourier2}, we obtain
\begin{equation}\label{eq:formula4}
\begin{split}
X_{\lambda_1,\dots,\lambda_n} &(v\otimes f)
=
\Res_{0}\!dz_{n-1}
\Res_{z_{n-1}}\!dz_{n-2}
\cdots
\Res_{z_2}\!dz_1
\\
& \times f(z_1,\dots,z_{n-1},0)
Y^{1\to\cdots\to n}\big(
e^{-\sum_{i=1}^{n-1}z_i(\lambda_i+\partial_i)}
v\big)
\,.
\end{split}
\end{equation}


In the case $n=1$, formulas \eqref{eq:formula3} and \eqref{eq:formula4}
reduce to $X_\lambda(v\otimes c)=c \,Y^{\bullet}(v)$,
where $c\in\mc O^{\star T}_1=\mb F$.
Finally, we consider the case $n=2$. In this case, $\mc L(2)$ consists only of the two graphs
$\bullet\,\,\,\bullet$ and $1\to 2$.
Hence, by \eqref{eq:formula3} and \eqref{eq:formula4}, we get
$$
X_{\lambda_1,\lambda_2}(v\otimes f)
=
f(w_1,w_2) * Y^{\bullet\,\,\bullet}_{\lambda_1,\lambda_2}(v)
+
\Res_{0}\!dz_{1}
f(z_1,0)
Y^{1\to2}_{\Lambda_1}\big(e^{-z_1(\lambda_1+\partial_1)}v\big)
\,.
$$
Note that $Y^{\bullet\,\,\bullet}$ has values in $V[\lambda_1,\lambda_2]/\langle\partial+\lambda_1+\lambda_2\rangle\simeq V[\lambda]$,
where we set $\lambda_1=\lambda$ and $\lambda_2=-\lambda-\partial$.
Hence, we denote its values as $Y^{\bullet\,\,\bullet}_\lambda(v)$.
Recall also that $Y^{1\to2}_{\Lambda_1}(v)$ is independent of $\Lambda_1$,
so we omit the subscript $\Lambda_1$.
Moreover, since $\mc O^{\star T}_2=\mb F[z_{12}^{\pm1}]$,
we may take $f(z_1,z_2)=z_{12}^m$, $m\in\mb Z$.
Under this setting, the previous formula can be rewritten as follows:
\begin{equation}\label{eq:formula5}
X_{\lambda}(v_1\otimes v_2\otimes z_{12}^m)
=
(-1)^m\partial_\lambda^m Y^{\bullet\,\,\bullet}_{\lambda}(v_1\otimes v_2)
+
(-1)^{m+1}
Y^{1\to2}\big((\lambda+\partial)^{(-m-1)}v_1\otimes v_2\big)
\,.
\end{equation}
As before, we are using the divided power notation:
$\lambda^{(-m-1)}=0$ for $m\geq0$ 
and $\lambda^{(-m-1)}=\lambda^{-m-1}/(-m-1)!$ for $m<0$.
Equation \eqref{eq:formula5}
agrees with the corresponding formulas in the proof of Theorem 10.10 of \cite{BDSHK}.

\subsection{Relation to the operad $\mathcal{L}ie$} \label{sec:lie-operad} 

Let $V = \mathbb{F}$ be the $1$-dimensional  vector space considered as an $\mathbb{F}[\partial]$-module 
with $\partial = 0$. We see from \eqref{eq:chiral-1}--\eqref{eq:chiral-2} that $P^{\ch}(2)$ is a $1$-dimensional 
vector space. Indeed any operation is determined by the image of $z_{12}^{-1} \in \mc O_n^{\star T}$. 
In fact, it follows from \cite[Eq.\ (6.25)]{BDSHK} that $P^{\ch}(2)$ is the non-trivial representation 
of the symmetric group $S_2$ on two elements. Let us call $\mu \in P^{\ch}(2)$ the operation such 
that $\mu \left( z_{12}^{-1} \right) = 1$.  Consider the operad $\mc Lie$ of Lie algebras, in which the vector 
space $\mc Lie(n)$ of $n$-ary operations has as a basis the set
\begin{equation} \label{eq:lie-basis}  
\bigl\{ [x_{\sigma(1)}, [x_{\sigma(2)}, [\cdots,x_{\sigma(n)}]\cdots]] \,\big|\, \sigma \in S_n, \; \sigma(1) = 1\bigr\} \,.
\end{equation}
In particular,
$\mc Lie(2)$ is the non-trivial $1$-dimensional representation of $S_2$, with a basis $[x_1, x_2]$. 
As an application of Theorem \ref{thm:main}, we obtain the following.
\begin{theorem}[{\cite[Sec.\ 3.1.5]{BD04}}]
There is a unique isomorphism of operads 
\[ 
P^{\ch}(\mathbb{F}) \simeq \mc Lie  \,,\quad \text{such that} \qquad  P^{\ch}(2) \ni \mu \mapsto [x_1,x_2] \in \mc Lie(2) \,. 
\]
\end{theorem}
\begin{proof}
By Theorem \ref{thm:main}, it is enough to prove the isomorphism 
of graded operads $P^{\cl}(\mathbb{F}) \simeq \mc Lie$.  
Let  $Y\in P^{\cl}(n)$ and $\Gamma \in \mc G(n)$. We see from \eqref{eq:sesq2} that $Y^\Gamma$ vanishes unless $\Gamma$ is connected, 
in which case $Y^\Gamma\colon\mathbb{F} \rightarrow \mathbb{F}$. It follows that  $P^{\cl}(n)$ is the quotient of $\mathbb{F} \mc G_c(n)$ by the cycle relations \eqref{eq:cycle},  where $\mc G_c(n)$ is the subset of connected graphs. 

The vector space $\mc Lie(n)$ has dimension $(n-1)!$ and has a basis given by \eqref{eq:lie-basis}. On the other hand, we see from Theorem \ref{lem:res5} that a basis for $P^{\cl}(n)$ is given by connected lines in $\mc L(n)$. It follows that $\dim P^{\cl}(n) = \dim \mc Lie(n) = (n-1)!$. The line 
\[
\begin{tikzpicture}
\node at (-0.2,1) {$\Gamma=$};
\draw (0.5,1) circle [radius=0.07];
\node at (0.5,0.6) {$1$};
\draw[->] (0.6,1) -- (1.2,1);
\draw (1.3,1) circle [radius=0.07];
\node at (1.3,0.6) {$\sigma(2)$};
\draw[->] (1.4,1) -- (2,1);
\node at (2.4,1) {$\cdots$};
\draw[->] (2.8,1) -- (3.4,1);
\draw (3.5,1) circle [radius=0.07];
\node at (3.5,0.6) {$\sigma(n)$};
%
\end{tikzpicture}
\]
is associated to the operation \eqref{eq:lie-basis}. One sees readily that this assignment is compatible with operations;  therefore, the map $\mc Lie(n) \rightarrow P^{\cl}(n)$ is an isomorphism. Uniqueness follows since any operation in $\mc Lie$ is a composition of binary operations. 
\end{proof}

\end{document}